\documentclass[12pt,a4paper]{article}

\usepackage{amssymb,amsmath,amsthm}
\usepackage[english]{babel}
\usepackage{t1enc}
\usepackage[latin2]{inputenc}
\usepackage{epsfig}
\newtheorem{thm}{Theorem}
\newtheorem{lem}[thm]{Lemma}

\newtheorem{claim}{Claim}
\DeclareMathOperator{\conv}{conv}
\newtheorem{prop}[thm]{Proposition}

\newcommand\cC{{\mathcal C}}

\newcommand\cF{{\mathcal F}}
\newcommand\cG{{\mathcal G}}
\newcommand\cH{{\mathcal H}}

\title{The profile polytope of non-trivial intersecting families}
\author{D\'aniel Gerbner\footnote{Alfr\'ed R\'enyi Institute of Mathematics, E-mail: \texttt{gerbner.daniel@renyi.hu.} Research supported by the
    National Research, Development and Innovation Office -- NKFIH under the
    grants FK 132060, KKP-133819, KH130371 and SNN 129364.}}

\date{}

\begin{document}

\maketitle

\newcommand{\inte}{\mathop{\mathrm{int}}\nolimits}
\newcommand{\co}{\mathop{\mathrm{co}}\nolimits}

\begin{abstract}

The profile vector of a family $\mathcal{F}$ of subsets of an
$n$-element set is $(f_0,f_1, \ldots, f_n)$ where $f_i$ denotes
the number of the $i$-element members of $\mathcal{F}$. In this
paper we determine the extreme points of the set of profile
vectors for the class of non-trivial intersecting families.

\end{abstract}

\section{Introduction and preliminaries}

Let $[n]=\{1,\dots,n\}$ be our underlying set. If $F\subseteq [n]$,
then $\overline{F}$ denotes the complement of $F$. Let
$\mathcal{F}$ be a family of subsets of $[n]$
(i.e. $\mathcal{F}\subseteq 2^{[n]}$). Let $\overline{\cF}:=\{F\subset [n]: \overline{F}\in\cF\}$. A family is called \emph{intersecting} if any
two members have non-empty intersection. Intersecting families of sets have attracted a lot of researchers, see e.g. Chapter 2 of the book \cite{book}. Let us start with a well-known and trivial statement.

\begin{prop} The maximum size of an intersecting family is
$2^{n-1}$.

\end{prop}

The maximum size is achieved e.g. by the family of all subsets containing a given fixed element. A family is called \emph{$k$-uniform}, if all its members have cardinality $k$. Let $\mathcal{F}_k$ denote the subfamily of the
$k$-element subsets in $\mathcal{F}$:\, $\mathcal{F}_k=\{F:
F\in\mathcal{F}, |F|=k\}$.

\begin{thm}[Erd\H{o}s, Ko, Rado \cite{ekr}] Let $k \le n/2$. Then the maximum size of a $k$-uniform intersecting family is $\binom{n-1}{k-1}$.

\end{thm}

Let us call an intersecting family \emph{trivial} if all its members contain a given fixed element, and non-trivial otherwise. The maximum in the above theorem is again achieved by the largest trivial intersecting family.

\begin{thm}[Hilton, Milner \cite{hm}] Let $k \le n/2$. Then the maximum size of a non-trivial $k$-uniform intersecting family is $1+\binom{n-1}{k-1}-\binom{n-k-1}{k-1}$.

\end{thm}

The maximum is given by the Hilton-Milner type family $HM(k)$, which we define next. $HM(k)$ contains $A=\{2,\dots,k+1\}$ and every $k$-element set which contains $1$ and intersects $A$. Moreover, Hilton and Milner \cite{hm} also showed that $HM(k)$ is the unique maximum if $3<k<n/2$, up to isomorphism (i.e. in every maximum family there is a fixed point contained in all but one of the members). If $k=3$, then there is another extremal family $\{F\in \binom{[n]}{3}: |F\cap [3]|\ge 2\}$.

We will also use the following generalized Hilton-Milner type families. Let $B=[m]$. Then $HM(k,m)=\{F \subset [n]: |F|=k, n \in F, | F \cap B|\ge 1\}$. This is a $k$-uniform, trivial intersecting family, but if we add $B$, it becomes a non-trivial intersecting family.

We will use the following simple observation.

\begin{lem}\label{observ} Let $m$ be the maximum cardinality of the members of a non-trivial intersecting family $\mathcal{F}$, let $i <m$ and $i\le n/2$. Then $|\mathcal{F}_i|\le |HM(i,m)|$. Moreover, if $i<n/2$, equality holds if and only if $\cF_i$ consists of all the sets containing a fixed element $x$ and intersecting an $m$-element set $B'$ with $x\not\in B'$ (i.e. $\cF_i$ is isomorphic to $HM(i,m)$).

\end{lem}

\begin{proof} If $\mathcal{F}_i$ is non-trivial, then $\mathcal{F}_i\le |HM(i,i)|< HM(i,m)$. If $\mathcal{F}_i$ is trivial, all its members contain a fixed element $x$. There is a set $F$ in $\cF$ which does not contain $x$ because of the non-triviality, and the $i$-element sets of $\cF$ also intersect $F$. One can easily see that $\mathcal{F}_i$ is a subset of a generalized Hilton-Milner type family then, and $HM(i,m)$ is the largest of those.
\end{proof}

We will use the Kruskal-Katona theorem \cite{kat,kru}. Given a $k$-uniform family $\cF\subset 2^{[n]}$, its \emph{shadow} is \[\Delta\cF:=\{G\subset [n]: \,|G|=k-1, \text{ there exists $F\in\cF$ with $G\subset F$}\}.\]
The \emph{shade} of $\cF$ is $\nabla\cF:=\{G\subset [n]: \,|G|=k+1, \text{ there exists $F\in\cF$ with $G\supset F$}\}$.
Given two sets $F,G\subset 2^{[n]}$, we say that $F$ is before $G$ in the \emph{colexicographical order} or \emph{colex order} if the largest element of the symmetric difference of $F$ and $G$ is in $G$. Let $\cC_k^\ell$ denote the family of the first $\ell$ sets from $\binom{[n]}{k}$ in the colex order. 

 Given two positive integers $\ell$ and $i$, there is a unique way to write $\ell$ in the form $\ell=\binom{n_i}{i}+\binom{n_{i-1}}{i-1}+\dots+\binom{n_j}{j}$ with $n_i>n_{1-i}>\dots>n_j\ge j\ge 1$. This form is called the \emph{cascade form} of $\ell$. The cascade form can be found in a greedy way: we pick the largest $n_i$ such that $\binom{n_i}{i}\le \ell$, then the largest $n_{i-1}$ such that $\binom{n_i}{i}+\binom{n_{i-1}}{i-1}\le \ell$, and so on.

The Kruskal-Katona shadow theorem \cite{kat,kru} states that if $\cF$ is a $k$-uniform family with $|\cF|=\ell$, then $|\Delta\cF|\ge |\Delta\cC_k^\ell|$. It is not hard to calculate the cardinality of $|\Delta\cC_k^\ell|$: if
$\ell=\binom{n_k}{k}+\binom{n_k-1}{k-1}+\dots+\binom{n_j}{j}$, then $|\Delta\cC_k^\ell|=\binom{n_k}{k-1}+\binom{n_k-1}{k-2}+\dots+\binom{n_j}{j-1}$. 

There is a simpler version of the shadow theorem due to Lov\'asz \cite{lov}. It states that if $\cF$ is a $k$-uniform family with $|\cF|=\binom{x}{k}$, then $|\Delta\cF|\ge \binom{x}{k-1}$. Here $x$ is not necessarily an integer and $\binom{x}{k}$ is defined to
be $\frac{x(x-1)\ldots (x-k+1)}{k!}$. This is a weaker bound, but easier to use. We will use both versions of the shadow theorem later.


\subsection{Profile polytopes}

The profile polytopes were introduced by P.L. Erd\H os, P. Frankl and
G.O.H. Katona in \cite{efk1}.
Recall that $\mathcal{F}_i$ denotes the subfamily of the
$i$-element subsets in $\mathcal{F}$. Its size $|\mathcal{F}_i|$ is denoted
by $f_i$. The vector ${\bf p}(\mathcal{F})=(f_0,f_1,\dots,f_n)$ in
the $(n+1)$-dimensional Euclidian space $\mathbb{R}^{n+1}$ is
called the \emph{profile} or \emph{profile vector} of $\mathcal{F}$.

If $\Lambda$ is a finite set in $\mathbb{R}^d$, its \emph{convex
hull} $\conv(\Lambda)$ is the set of all convex combinations of
the elements of $\Lambda$. A point of $\Lambda$ is an
\emph{extreme point} if it is not a convex combination of other
points of $\Lambda$. It is easy to see that the convex hull of a
set is equal to the convex hull of the extreme points of the set.

Let $\mathbf{A}$ be a class of families of subsets of $[n]$. We
denote by $\Lambda(\mathbf{A})$ the set of profiles of the
families belonging to $\mathbf{A}$:

\[\Lambda(\mathbf{A})=\{{\bf p}(\mathcal{F}):
\mathcal{F}\in\mathbf{A}\}.\]

The \emph{profile
polytope} of $\mathbf{A}$ is $\conv(\Lambda(\mathbf{A}))$.
We are interested in the extreme points of $\Lambda(\mathbf{A})$. We simply call them the extreme points of
$\mathbf{A}$.







Suppose we are given a weight function
$w:\{0,\dots,n\}\rightarrow\mathbb{R}$, and the weight of a family
$\mathcal{F}$ is defined to be $\sum_{F\in\mathcal{F}} w(|F|)$,
which is equal to $\sum_{i=0}^n w(i)f_i$. Usually we are
interested in the maximum of the weight of the families in a class
$\mathbf{A}$. So we want to maximize this sum, i.e. find a family
$\mathcal{F}_0 \in \mathbf{A}$ and an inequality $\sum_{i=0}^n
w(i)f_i=w(\mathcal{F})\le w(\mathcal{F}_0)=c$. This is a linear
inequality, and it is always maximized in an extreme point.

Given a class (or property) of families, the first natural question in extremal combinatorics is the maximum cardinality such a family can have. When it is answered, often some simple weight functions are considered and the maximum weight of such a family is studied. Determining the extreme points answers these questions for every (linear) weight function.


P.L. Erd\H os, P. Frankl and
G.O.H. Katona \cite{efk1} determined the
extreme points of the intersecting Sperner
families. In their next paper \cite{efk2}, the extreme points of
the profile polytope of the intersecting families were determined. Now we define these. Let coordinate $i$ of ${\bf a}$ be $0$ if $i<n/2$, $\binom{n-1}{i-1}$ if $i=n/2$ and $\binom{n}{i}$ if $i>n/2$. Let $k \le n/2$. Coordinate $i$ of ${\bf a}_k$ is $0$ if $i<k$, $\binom{n-1}{i-1}$ if $k \le i \le n-k$, and $\binom{n}{i}$ if $i>n-k$. Let $\Gamma_a$ be the set of vectors that we can get from any of the vectors ${\bf a}_k$ and ${\bf a}$, if we replace an arbitrary set of coordinates by $0$. Note that if $n$ is even, then ${\bf a}={\bf a_{n/2}}$.

\begin{thm}[P.L. Erd\H{o}s, Frankl, Katona \cite{efk2}]\label{metszo} The set of extreme points of the intersecting families is $\Gamma_a$.
\end{thm}

The corresponding intersecting families are the following. $\mathcal{A}_k$ consists of the sets which have sizes at least $k$ and contain the element $n$, and of every other set which has size greater than $n-k$. $\mathcal{A}$ consists of all the sets with size greater than $n/2$, and the sets which have sizes $n/2$ and contain $n$. These families are obviously intersecting and their profile vectors are ${\bf a}_k$ and ${\bf a}$. We can delete full levels and the families are still intersecting; in the corresponding vectors some coordinates are changed to $0$.

Since then several other classes of families have been considered, see e.g. \cite{eng}, \cite{gerbner}, generalizations have been studied \cite{patk}, \cite{patk2}, and profile polytopes were applied for counting subposets \cite{gkp}. Note that most of the classes of families where the profile polytope has been studied are \emph{hereditary}, i.e. if we remove some members of a family in the class, the resulting family still belongs to the class. It makes determining the extreme points easier, as we do not have to deal with negative weights, and all extreme points can be achieved by changing some coordinates of a few essential ones to $0$.
However, in this paper we determine the extreme points of the non-trivial intersecting families, which is not a hereditary property. 

In the next section we define what is needed to state our main theorem. We prove an important special case in Section \ref{biz}, and finish the proof by a case analysis in Section \ref{mainbiz}.

\section{The main theorem}

Let us start with some simple observations. A non-trivial intersecting family cannot contain the empty set or a singleton. It might contain the full set, but that does not change the intersecting property, nor the nontrivial property. It means that for a weight function $w$ if $w(n)>0$, the maximum family contains the full set, if $w(n)<0$, it does not. Moreover, changing only $w(n)$ does not change the other parts of the maximum family, hence we can basically forget about $n$. More precisely, $(p_0,p_1, \dots, p_{n-2}, p_{n-1},0)$ is an extreme point if and only if $(p_0,p_1, \dots, p_{n-2}, p_{n-1},1)$ is an extreme point.

Now we define several vectors, which are going to be the extreme points of the non-trivial intersecting families. Then we state our main theorem, and after that we show that these vectors indeed correspond to non-trivial intersecting families and are extreme points (note that for most classes of families where profile polytopes have been studied, these statements are trivial, but not for the non-trivial intersecting families). That part also makes it easier to understand where these definitions come from. All these vectors are in the $(n+1)$-dimensional Euclidean space, but coordinates 0,1 and $n$ are always 0. Let $H\subset \{2,3,\dots,n-2,n-1\}$ be a nonempty set of indices, $h$ be its smallest and $h'$ be its largest element.

Let ${\bf b}_H=(b_0,\dots,b_n)$ with
\begin{displaymath}
b_i=
\left\{ \begin{array}{l l}
0 & \textrm{if\/ $i\not\in H$},\\
|HM(i,h')| & \textrm{if\/ $i\in H$ and $i< h'$},\\
|HM(i,h')|+1 & \textrm{if\/ $i=h'$}.\\
\end{array}
\right.
\end{displaymath}
Let $\Gamma_b=\{ {\bf b}_H: h+h' \le n\}$.

Let ${\bf c}_H=(c_0,\dots,c_n)$ with  
\begin{displaymath}
c_i=
\left\{ \begin{array}{l l}
0 & \textrm{if\/ $i\not\in H$},\\
\binom{n-1}{i-1} & \textrm{if\/ $i\in H$ and $i\le n-h'$},\\
\binom{n}{i} & \textrm{otherwise}.\\
\end{array}
\right.
\end{displaymath}
Let $\Gamma_c=\{ {\bf c}_H: h+h' > n\}$.

Let ${\bf d}_H=(d_0,\dots,d_n)$ with
\begin{displaymath}
d_i=
\left\{ \begin{array}{l l}
0 & \textrm{if\/ $i\not\in H$},\\
|HM(i,h')| & \textrm{if\/ $i\in H$ and $i< h'$},\\
1 & \textrm{if\/ $i=h'$}.\\
\end{array}
\right.
\end{displaymath}
Let $\Gamma_d=\{ {\bf d}_H: |H|>1, h+h''\le n\}$, where $h''$ is the second largest element of $H$.

Let us consider the set $P$ of vectors $(e_0,\dots,e_n)$ satisfying the following properties.

1, Every $e_i$ is a non-negative integer, $e_0=e_1=e_n=0$.

2, $x:=\sum_{i=2}^{n-1}e_i \ge 3$.

3, $\sum_{i=2}^{n-1} ie_i \le (x-1)n$.

Now we show the connection between $P$ and non-trivial intersecting families. For two vectors ${\bf p}=(p_0,\dots,p_n)$ and ${\bf p'}=(p_0',\dots,p_n')$, we say that ${\bf p'} \le {\bf p}$ if $p_i\le p_i'$ for every $0\le i\le n$.

\begin{lem}\label{ujabb}

\textbf{(i)} If a non-trivial intersecting family does not contain $[n]$, its profile is in $P$.

\textbf{(ii)} If ${\bf p} \in P$ and there is no ${\bf p'} \in P$ different from ${\bf p} $ with ${\bf p'} \le {\bf p}$, then ${\bf p}$ is the profile of a non-trivial intersecting family.

\end{lem}

\begin{proof}

To show \textbf{(i)}, observe that for the profile of a non-trivial intersecting family obviously $e_0=e_1=0$ holds, and also we need at least three members in the family, as any two members trivially intersect. The third property is needed, otherwise an element of the underlying set would be covered $x$ times, i.e. by every set, contradicting the non-triviality.

Let us prove now \textbf{(ii)}. We are given a vector ${\bf p}$ and we are going to construct a non-trivial intersecting family $\cF$ with profile ${\bf p}$. Observe that ${\bf p}$ shows how many $k$-element sets must be in the family for every $k$. Let us denote the sizes of the sets by $a_1,\dots,a_\ell$ in decreasing order. We choose the first (the largest) set $F_1$ of size $a_1$ arbitrarily. Let $B_i$ be the set of vertices which are not covered by each of the first $i$ sets $F_1,\dots,F_i$ (only by at most $i-1$ of them), then $B_1=\overline{F_1}$ and $B_i\supset B_{i-1}$ for every $i>1$. We choose the second set $F_2$ of size $a_2$ in such a way that $F_2$ intersects $F_1$ and also $F_2$ contains $B_1$, if possible.

If it is not possible, then we claim that we have $x=3$. Indeed, in that case we have $a_1+a_2\le n$, thus they together with the next set $F_3$ of size $a_3$ have their profile in $P$, which means no other set can be in the family because of our assumption on the minimality of ${\bf p}$. 
Then we pick $F_2$ of size $a_2$ such that it intersects $F_1$ in a single element, and then we pick $F_3$ of size $a_3$ such that it contains an element of $F_1\setminus F_2$ and an element of $F_2\setminus F_1$. This is doable as $a_3\ge 2$. The resulting family is clearly non-trivial intersecting.

If $x>3$, we choose every $F_i$ of size $a_i$ in such a way that it contains $B_{i-1}$, if possible. Note that in this case it automatically intersects $F_1, \dots, F_{i-1}$. Indeed, $F_i$ contains $B_1$, which is also contained in $F_2, \dots, F_{i-1}$. $F_i$ also contains $B_2$, which intersects $F_1$ (we also use that $B_1$ and $B_2$ are not empty).


Now assume that when we add a set $F_i$, it is too small to cover every vertex in $B_i$, i.e. $a_i< |B_i|$. Then $i=\ell$, i.e. $F_i$ is the last set (as the resulting profile vector is in $P$). We have to choose $F_i$ in such a way that it intersects the other sets. As every vertex is covered at least $i-1$ times, all we have to do is to put an arbitrary vertex of $B_{i-1}$ in $F$, then the new set intersects all but one of the earlier sets, say $F_j$. We have to choose a vertex in $B_{i-1}$ contained in $F_j$, and then other vertices from $B_{i-1}$ arbitrarily. As only vertices in $B_{i-1}$ are used, no vertex is covered $i$ times, hence the family is non-trivial.
\end{proof}

Let $\Gamma_e$ be the set of the extreme points of $P$. Now we can state our main theorem.








\begin{thm}\label{main}

The extreme points of the profile polytope of the non-trivial intersecting families are the elements of $\Gamma_b\cup\Gamma_c\cup\Gamma_d\cup\Gamma_e$, and additionally the vectors we get from these if we change the last coordinate from $0$ to $1$.



\end{thm}

To prove this statement, we have to show that the points listed are indeed extreme points, and that there are no other extreme points. The first part is the easier task, and we will deal with it in the rest of this section.
We give an example non-trivial intersecting family for each of the vectors ${\bf v}={\bf v}_H\in \Gamma_b\cup\Gamma_c\cup\Gamma_d\cup\Gamma_e$ and also show that ${\bf v}$ is an extreme point, by showing a weight function such that ${\bf v}$ is the unique maximum. 

Let us describe first the general approach to find such a weight function.
We start by assuming that if $i\not\in H$, then $w(i)$ is negative, moreover, $w(i)$ so small compared to the other weights $w(j)$, that if a family contains even one $i$-element set, its total weight is negative. On the other hand, there is a $1<j<n$ with $w(j)>0$, thus there is a family of positive weight. This shows that
no $i$-element sets can be in the family of maximum weight. Similarly, we can say that for some $i\in H$ its weight is very large compared to the other weights. It implies that the family of maximum weight contains as many $i$-element sets as possible, i.e. $|HM(i,j)|$, where $j$ is the largest non-zero coordinate of ${\bf v}_H$. We describe these ideas in more details in the proof of the following lemma. 


\begin{lem}\label{gammab} The elements of $\Gamma_b$ are extreme points of the non-trivial intersecting families.

\end{lem}

\begin{proof} For ${\bf b}_H\in\Gamma_b$ we have to show a family $\mathcal{B}_H$ which has ${\bf b}_H$ as profile, and a weight $w$ which is maximized at ${\bf b}_H$. Let $\mathcal{B}_H=[h']\cup\left(\bigcup_{i\in H}HM(i,h')\right)$, i.e. the union of $HM(i,h')$ for every $i\in H$, and additionally $[h']$. This family is obviously non-trivial intersecting, as each of its members except for $[h']$ contains $n$ and intersects $[h']$.

Now we are going to show a weight function that is maximized only by families with profile ${\bf b}_H$.
Let $w$ be a weight such that if $i\not\in H$, then $w(i)=-2^{2n}$. It is going to be so small compared to the other weights, that no $i$-element sets can be in the maximum family $\cF$. All other sets have weight at most $2^n$, and there are less than $2^n$ sets in $\cF$, hence positive weight can only be achieved without these negative sets. Let $w(h)=2^n$, it is very large compared to the other positive weights (but still very small compared to the absolute value of the negative weights), and all other weights are 1. Then a single $h$-element set has larger weight than all the other sets with positive weight, thus the maximum family $\cF$ contains as many $h$-element sets as possible. If $h<n/2$, then by Lemma \ref{observ} the maximum number of $h$-element sets is $|HM(h,h')|$, and the largest family of $h$-element sets is isomorphic to $HM(h,h')$. Without loss of generality we can assume that $\cF_h$ is equal to $HM(h,h')$. 

Observe that the only set of size at most $h'$ that intersects every member of $HM(h,h')$ is $[h']$. Therefore, every other member of $\cF$ should contain the fixed point $n$ except for $[h']$.
Also, every member of $\cF$ should intersect $[h']$, hence $\cF$ is a subfamily of $\mathcal{B}_H$. Then $\mathcal{B}_H$ is the unique maximum.

Finally, if $h=n/2$, then the only non-zero coordinate is $\binom{n-1}{n/2-1}$ at coordinate $n/2$. This vector is an extreme point of the class of intersecting families, thus it is an extreme point of this smaller family as well.
\end{proof}

\begin{lem} The elements of $\Gamma_c$ are extreme points of the non-trivial intersecting families.

\end{lem}

\begin{proof} These are the elements of $\Gamma_a$ which correspond to non-trivial intersecting families. They are extreme points of the larger set (of all intersecting families), thus they are extreme points of the smaller set as well.
\end{proof}

\begin{lem} The elements of $\Gamma_d$ are extreme points of the non-trivial intersecting families.

\end{lem}

\begin{proof}

For ${\bf d}_H\in \Gamma_c$ we define the family $$\mathcal{D}_H=\cup_{i\in H,\, i\neq h'} HM(i,h')\cup \{[h']\}.$$ It is the same as $\mathcal{B}_H$, except we removed most of the $h'$-element sets. Let $w$ be the weight function described in the proof of Lemma \ref{gammab}. We use almost the same weight here, we set $w'(h')=-1$ and $w'(i)=w(i)$ for every $i \neq h'$. Just as in  Lemma \ref{gammab}, we need to have the largest number of $h$-element sets in the family $\cF$ of maximum weight, that is $|HM(h,h')|$. Then we need an $h'$-element set in $\cF$. Without loss of generality, $\cF_h=HM(h,h')$ and $\cF_{h'}\supset\{[h']\}$. However, there is no point in having more $h'$-element sets in $\cF$. For every other $i\in H$, there are at most $|HM(i,h')|$ sets of size $i$ in $\cF$, completing the proof.
\end{proof}

\begin{lem}\label{negati} The elements of $\Gamma_e$ are extreme points of the non-trivial intersecting families.

\end{lem}

\begin{proof}
Observe first that if ${\bf v}=(e_0,\dots, e_n)\in P$ and we increase $e_i$, the resulting vector is in $P$, as we cannot violate any of the properties. It means that the extreme points of $P$ are minimal. More precisely if ${\bf e}$ is an extreme point and ${\bf e'} \ge {\bf e}$ (with ${\bf e'} \neq {\bf e}$), then ${\bf e'}$ cannot be an extreme point, as it is a convex combination of the following two elements of $P$: ${\bf e}$ and $2{\bf e'}-{\bf e}$. By Lemma \ref{ujabb}, ${\bf e}$ is the profile vector of a non-trivial intersecting family.

We need to show that the extreme points of $P$ are extreme points of the non-trivial intersecting families as well. Let $P'$ be the set of the profile vectors of those non-trivial intersecting families which do not contain $[n]$. Then $P' \subset P$ by Lemma \ref{ujabb}.

An extreme point ${\bf p}$ of the larger set $P$ which also belongs to the smaller set $P'$ must be obviously extreme point of $P'$ as well. Thus there exists a weight function $w$ where it gives the maximum in $P'$. One can easily see that if we change $w(n)$ to a negative number, then ${\bf p}$ has the largest weight among every profile vectors of non-trivial intersecting families.
\end{proof}

\section{The main lemma}\label{biz}










Our most important special case is when there are only two non-empty levels $1<i< m<n$ with $m> n/2$ and $i+m \le n$. For other values of $i$ and $m$, it is going to be easy to see that Theorem \ref{main} holds (we do it inside the proof of the main theorem in Section \ref{mainbiz}). Thus, the lemma below contains the most complicated part of the proof.

\begin{lem}\label{ketto} Let $(f_0, f_1, f_2, \dots, f_n)$ be the profile vector of a non-trivial intersecting family $\mathcal{F}$. Let us assume that $m$ is the maximum cardinality in $\mathcal{F}$, $m> n/2$ and $i+m \le n$. Then there is a $0\le \lambda\le 1$ such that $(f_i,f_m) \le \lambda (0, \binom{n}{m})+(1-\lambda)(|HM(i,m)|,|HM(m,m)|)$.
\end{lem}


We will use the following simple observations.

\begin{prop}\label{triv}
\textbf{(i)} If $x\le y$, then $\binom{x}{k-1}/\binom{x}{k}\ge \binom{y}{k-1}/\binom{y}{k}$.

\textbf{(ii)} Let $0\le c'$, $0<\alpha,a,b,c,b'$ with $bc'\le cb'$, $b/c\le\alpha$ and $c\ge c'$. Then 
\[\frac{\alpha a+b}{a+c}\le \frac{\alpha a+b'}{a+c'}.\]
\end{prop}


\begin{proof} The first statement easily follows from the definition of $\binom{x}{k}$.


By rearranging the desired inequality of $\textbf{(ii)}$, we obtain the equivalent form  $\alpha ac'+ab+bc'\le \alpha ac+ab'+cb'$. Recall that we have $bc'\le cb'$. The other terms can be rewritten as $\frac{b-b'}{c-c'}\le\alpha$. We have $\frac{b-b'}{c-c'}\le\frac{b-bc'/c}{c-c'}=\frac{b(c-c')/c}{c-c'}=b/c\le\alpha$.
\end{proof}

Now we are ready to prove Lemma \ref{ketto}.

\begin{proof}[Proof of Lemma \ref{ketto}]

We use induction on $n-m-i$. Observe that for the base case $i+m=n$ we have that $HM(i,m)\cup HM(m,m)$ consists of all the $i$-sets and $m$-sets containing $n$, except that it contains $[m]$ instead of its complement. Thus $HM(i,m)\cup HM(m,m)$ has $\binom{n}{i}$ members, just like any maximal non-trivially intersecting family on these two levels. Let us choose $\lambda=\frac{|HM(i,m)|-f_i}{|HM(i,m)|}$, then by definition $f_i\le (1-\lambda)|HM(i,m)|$, and we need \[f_m\le \lambda\binom{n}{m}+(1-\lambda)|HM(m,m)|=\binom{n}{m}-\frac{f_i\binom{n}{m}}{|HM(i,m)|}+\frac{f_i|HM(m,m)|}{|HM(i,m)|}=\binom{n}{m}-f_i.\]
This holds for every intersecting family, even the trivial one. For non-trivial intersecting families, we have $f_i\le |HM(i,m)|$ by Lemma \ref{observ}, thus we have $\lambda\ge 0$, completing the proof of the base step.

Let us continue with the induction step. Let us consider $\nabla\cF_i$, which is the shade of $\cF_i$ and let $g_{i+1}=|\nabla\cF_i|$. Then $\nabla\cF_i\cup\cF_m$ is obviously non-trivially intersecting, thus by the induction hypothesis there is a $0\le \lambda\le 1$ such that $(g_{i+1},f_m) \le \lambda (0, \binom{n}{m})+(1-\lambda)(|HM(i+1,m)|,|HM(m,m)|)$. We will show that the same $\lambda$ works for $f_i$, i.e. $(f_i,f_m) \le \lambda (0, \binom{n}{m})+(1-\lambda)(|HM(i,m)|,|HM(m,m)|)$. As the values in coordinate $m$ do not change, all we need to prove is that $f_i\le (1-\lambda)|HM(i,m)|$ if $g_{i+1}\le (1-\lambda)|HM(i+1,m)|$. It is enough to show  that$f_i/|HM(i,m)|\le g_{i+1}/|HM(i+1,m)|$, or equivalently $g_{i+1}/f_i\ge |HM(i+1,m)|/|HM(i,m)|$. As $HM(i+1,m)=\nabla HM(i,m)$,
the last of the above inequalities means that the size of the shade of $\cF_i$ is proportionally the smallest if $\cF_i$ is $HM(i,m)$.

We will use the Kruskal-Katona theorem. To use it in the form we have stated it, we will consider the complement family, as the shade of a family is the shadow of its complement.

Observe that $\overline{HM(i,m)}$ is an initial segment of the colex ordering if we reorder the elements of $[n]$. Indeed, members of $\overline{HM(i,m)}$ completely avoid a given element $z$, and then we take all the $(n-i)$-sets but those that contain an $m$-element set $B$. By reordering, we can assume that $z=n$ and $B=\{n-m,\dots,n-1\}$. The sets containing $n$ are the last in the colex order, and a superset $F$ of $B$ cannot be before a set $G\in\overline{HM(i,m)}$, as the largest element of $F\setminus G$ is in $B$, while every element of $G\setminus F$ is less than $n-m$.

The cascade form of $|\overline{HM(i,m)}|$ is $\binom{n-2}{n-i}+\binom{n-3}{n-i-1}+\binom{n-4}{n-i-2}+\dots+\binom{n-m}{n-i-m+2}=\sum_{j=2}^m\binom{n-j}{n-i-j+2}$. Let $\cG$ be a non-empty $(n-i)$-uniform family 
with $|\cG|< |\overline{HM(i,m)}|$ and cascade form $|\cG|=\sum_{j=2}^{m'}\binom{n_j}{n-i-j+2}$. Observe that $n_2\le n-2$. This implies that for any $h$, $n_h\le n_h$.

We partition $\overline{HM(i,m)}$ into $m-1$ parts: $\cH_2$ consists of the first $\binom{n-2}{n-i}$ sets of $\overline{HM(i,m)}$ in the colex order, $\cH_3$ consists of the next $\binom{n-3}{n-i-1}$ sets, and so on. $\cH_j$ for $j\le m$ consists of $\binom{n-j}{n-i-j+2}$ sets that come after $\cH_2,\dots, \cH_{j-1}$, i.e. after the first $\binom{n-2}{n-i}+\binom{n-3}{n-i-1}+\binom{n-4}{n-i-2}+\dots+\binom{n-m}{n-i-m+2}$ sets in the colex order.
We also partition $\cG$ into $m-1$ parts: for $2\le j<m$, $\cG_j$ similarly consists of $\binom{n_j}{n-i-j+2}$ sets of $\cG$ that come after $\cG_2,\dots, \cG_{j-1}$ in the colex order. Then $\cG_m$ consists of all the remaining $\sum_{j=m}^{m'}\binom{n_j}{n-i-j+2}$ sets of $\cG$. Let us note that $\cG_j$ can be empty if $j>2$.

Let us assume that $n_2=n-2$, $n_3=n-3$,...,$n_{h}=n-(h)$ and $n_{h+1}< n-h-1$. 
Let $\cH^*=\cup_{j=1}^h \cH_j$, $\cH^{**}=\cup_{j=h+1}^m \cH_j$, $\cG^*=\cup_{j=1}^m \cG_j$, $\cG^{**}=\cup_{j=h+1}^{m'} \cG_j$. Observe that we have $|\cH^*|=|\cG^*|$ and $|\Delta\cH^*|\le |\Delta \cG^*|$ since $\cH^*$ is an initial segment of the colex ordering. We also have $|\cH^{**}|\ge \binom{n-h-1}{n-i-h+1}$ and $|\cG^{**}|<\binom{n-h-1}{n-i-h+1}$.

Let $a:=|\cH^*|$, $c:=|\cH^{**}|$, $\alpha=|\Delta\cH^*|/|\cH^*|$, $b:=|\Delta\cH^{**}\setminus\Delta\cH^*|$, $b'=|\Delta\cG|-|\Delta\cG^*|$, $c':=|\cG^{**}|$ and $\alpha'=|\Delta\cG^*|/|\cG^*|$. Our goal is to apply \textbf{(ii)} of Proposition \ref{triv}. By the above, we have $c>c'$. Now we will show that the other conditions are satisfied as well.

We let $p_\ell:=\binom{n-\ell}{n-i-\ell+2}=|\Delta\cH_\ell\setminus\Delta\bigcup_{\ell'=2}^{\ell-1}\cH_{\ell'})|$, i.e. the number of sets added to the shadow of $\bigcup_{\ell'=2}^{\ell}\cH_{\ell'}$ by $\cH_\ell$.
Observe first that $p_\ell/|\cH_\ell|=(n-i-\ell+2)/(i-1)$, thus $p_\ell/|\cH_\ell|$ decreases as $\ell$ increases. This implies that
$p_\ell/|\cH_\ell|\le p_{h+1}/|\cH_{h+1}|$ for every $\ell>h+1$. Therefore, we have that \begin{equation}\label{ineq0} \frac{b}{c}\frac{|\Delta\cH^{**}\setminus\Delta\cH^*|}{|\bigcup_{\ell=h+1}^{m} \cH_\ell|}=\frac{\sum_{\ell=h+1}^{m}p_\ell}{|\bigcup_{\ell=h+1}^{m} \cH_\ell|}\le \frac{\frac{p_{h+1}}{|\cH_{h+1}|}|\bigcup_{\ell=h+1}^{m} \cH_\ell|}{|\bigcup_{\ell=h+1}^{m} \cH_\ell|}=\frac{p_{h+1}}{|\cH_{h+1}|}.
\end{equation}

Similarly, we have that \[\alpha=\frac{|\Delta\cH^*|}{|\cH^*|}=\frac{|\Delta\cup_{j=1}^h \cH_j|}{|\cup_{j=1}^h \cH_j|}=\frac{\sum_{j=1}^{h}p_j}{|\bigcup_{j=1}^{h} \cH_j|}\ge \frac{\frac{p_{h+1}}{|\cH_{h+1}|}|\bigcup_{j=1}^{h} \cH_j|}{|\bigcup_{j=1}^{h} \cH_j|}=\frac{p_{h+1}}{|\cH_{h+1}|}\ge\frac{b}{c}, \]
where the last inequality uses (\ref{ineq0}).

Let $x< n-h-1$ be defined by $\binom{x}{n-i-h+1}:=\binom{n_{h+1}}{n-i-h+1}+\binom{n_{h+2}}{n-i-h}+\ldots+\binom{n_{m'}}{n-i-m'+2}=|\bigcup_{\ell=h+1}^{m'} \cG_\ell|$. We have $|\Delta\cG|\ge |\Delta\cH^*|+\binom{n_{h+1}}{n-i-h}+\binom{n_{h+2}}{n-i-h-1}+\ldots+\binom{n_{m'}}{n-i-m'+1}$ by the Kruskal-Katona theorem. We claim that 

\begin{equation}\label{ineq1} \binom{n_{h+1}}{n-i-h}+\binom{n_{h+2}}{n-i-h-1}+\ldots+\binom{n_{m'}}{n-i-m'+1}\ge \binom{x}{n-i-h}.
\end{equation}

Indeed, the left hand side is the sharp lower bound on the size of the shadow of an $(n-i-h+1)$-uniform family of size $\binom{x}{n-i-h+1}$ by the Kruskal-Katona theorem, while the right hand side is the not necessarily sharp lower bound on the size of the  same family by Lov\'asz's version of the shadow theorem. We have $\frac{b'}{c'}=\frac{\binom{n_{h+1}}{n-i-h}+\binom{n_{h+2}}{n-i-h-1}+\ldots+\binom{n_{m'}}{n-i-m'+1}}{\binom{x}{n-i-h+1}}\ge\frac{\binom{x}{n-i-h}}{\binom{x}{n-i-h+1}}\ge \frac{\binom{n-h-1}{n-i-h}}{\binom{n-h-1}{n-i-h+1}}=p_{h+1}/|\cH_{h+1}|\ge \frac{b}{c}$. In the inequalities here we used (\ref{ineq1}) first, then \textbf{(i)} of Proposition \ref{triv} and finally (\ref{ineq0}). 

Now we can apply \textbf{(ii)} of Proposition \ref{triv} to show that
$\frac{\alpha a+b}{a+c}\le \frac{\alpha a+b'}{a+c'}\le \frac{\alpha' a+b'}{a+c'}$.
This means
$|\Delta\overline{HM(i,m)}|/|\overline{HM(i,m)}|\le|\Delta \cG|/|\cG|$. By taking the complements, we obtain that $|\nabla HM(i,m)|/|HM(i,m)|\le |\nabla \cG'|/|\cG'|$ for any $i$-uniform family $\cG'$ with $|\cG'|\le |HM(i,m)|$. In particular, $|\nabla HM(i,m)|/|HM(i,m)|\le g_{i+1}/f_i$, completing the proof.

\end{proof}

\section{Proof of the main theorem}\label{mainbiz}

In this section we finish the proof of Theorem \ref{main}.
It is easy to see that we can consider only families not containing $[n]$.
It is enough to show that if a profile vector ${\bf p}$ of a non-trivial intersecting family $\mathcal{F}$ gives the unique maximum for a weight function $w$, then ${\bf p}\in \Gamma_b \cup \Gamma_c \cup \Gamma_d\cup \Gamma_e$.

An important observation is that if $F\in\cF$, $F \subset G$ and $G$ has positive weight, then $G$ is in the maximum family (as adding it would not violate any of the properties). In the proof we often start with fixing the maximum size $m$ of members, it implies that larger sets (except possibly $[n]$) do not have positive weight. Note that if $w(m)>0$, then $\cF_m$ is non-trivial intersecting. Indeed, if $\cF_m$ is trivial, then all its members contain a given element $x$ and there is a set $F\in\cF$ of smaller size not containing $x$. But then all the $m$-element sets which contain $F$ are in $\cF$, even those which do not contain $x$, a contradiction.


We continue the proof with a case analysis.

{\bf Case 1. $w(i)\le 0$ for every $1<i<n$.}

\medskip

{\bf Case 1a. $w(i)< 0$ for every $1<i<n$.}

Obviously ${\bf p \in P}$, as in all the cases, by Lemma \ref{ujabb}. But in this special case we will show that ${\bf p}$ is also an extreme point of $P$, thus it is in $\Gamma_e$. 

If ${\bf p}$ is not an extreme point of $P$, i.e. there is an element ${\bf p'}$ of $P$ with larger weight, then either ${\bf p'}$ also corresponds to a non-trivial intersecting family (a contradiction), or there is another ${\bf p''}\in P$ with ${\bf p''} \le {\bf p'}$ by Lemma \ref{ujabb}. But then ${\bf p''}$ has even larger weight. As ${\bf p''}$ cannot correspond to a non-trivial intersecting family, there is an even smaller vector in $P$ (with even larger weight). It cannot continue forever, as each coordinate is a non-negative integer. We arrive to a vector which does correspond to a non-trivial intersecting family, hence has larger weight than ${\bf p}$, a contradiction.

\smallskip
{\bf Case1b. $w(i)=0$ for some $1<i<n$.}

Obviously $w({\bf p})\le 0$, and $w(HM(i))=0$. The profile of $HM(i)$ is in $\Gamma_b$.

\bigskip

{\bf Case 2. $w(i)>0$ for some $1<i<n-1$.}

Let $m$ be the maximum size in $\mathcal{F}$. We will use Lemma \ref{observ} several times.

\medskip

{\bf Case 2a. $w(m)<0$.}

There is an $m$-element set $F$ in $\mathcal{F}$. Obviously the only reason it is in the family is that without it the family would be trivial, hence every other member of $\mathcal{F}$ contains a fixed point. Then for every level $i$ the maximum weight is given either by the empty family (in case $w(i)\le 0$), or $HM(i,m)$. Then we take the union of these uniform families (on every level $i$, the empty family of $HM(i,m)$), and we add $[m]$. The resulting family is non-trivial intersecting, and its profile vector is in $\Gamma_d$.

\medskip

{\bf Case 2b. $w(m)\ge 0$ and $m \le n/2$.}

For every level below $m$, the maximum of $w(\mathcal{F}_i)$ is either $0$ or given by $HM(i,m)$. In particular, on level $m$ clearly $HM(m)=HM(m,m) \cup \{[m]\}$ gives the maximum weight. This already makes sure the family is non-trivial intersecting, hence for every other level $j$ with $j<m$ we can choose $HM(j,m)$ or the empty family, depending on whether $w(j)$ is positive or not. The union of these uniform families is non-trivial intersecting, and has the largest possible weight on every level up to $m$. Its profile is in $\Gamma_b$.



\medskip

{\bf Case 2c. $w(m)\ge 0$ and $m>n/2$.}







Let $m_0$ be the size of the smallest member of the family $\mathcal{F}$.

\smallskip

{\bf Case 2c1. $w(m)\ge 0$, $m>n/2$ and $m+m_0> n$.}

Let us consider the following modified weight function. Let $w'(i)$ be the same as $w(i)$ if $m_0 \le i \le m$ and negative otherwise. Obviously the maximum non-trivial intersecting family for $w'$ is also $\mathcal{F}$. Let us examine the intersecting family $\cF'$ with maximum weight $w'$ now. One can easily see using Theorem \ref{metszo} that the profile of $\cF'$ can be obtained from ${\bf a_{m_0}}$ by changing some coordinates to 0. If $w(m)=0$, then $\cF'$ might contain no $m$-element sets, but even in this case we can add every $m$-element set to $\cF'$ without decreasing the weight (and without ruining the intersecting property). The resulting family $\mathcal{F}''$ is non-trivial intersecting, and $w'(\mathcal{F}'')=w'(\mathcal{F}') \ge w'(\mathcal{F}) \ge w(\mathcal{F})$, thus $\mathcal{F}''$ must have the same profile as $\mathcal{F}$. The profile of $\mathcal{F}''$ is in $\Gamma_c$.

\smallskip

{\bf Case 2c2. $w(m)\ge 0$, $m>n/2$ and $m+m_0 \le n$.}

Let $H$ be the set of non-empty levels. Recall that coordinate $i$ of ${\bf a}$ is $0$ if $i<n/2$, $\binom{n-1}{i-1}$ if $i=n/2$ and $\binom{n}{i}$ if $i>n/2$. Let ${\bf a'}$ be the vector we get from ${\bf a}$ when we change the coordinates not in $H$ to 0. We will show that ${\bf p}={\bf b}_H$, by showing that there is a $\lambda$ such that $\lambda {\bf b}_H + (1-\lambda){\bf a'} \ge {\bf p}$. We have that ${\bf b}_H$ and ${\bf a'}$ are both 0 in the negative coordinates, thus the weight of either ${\bf b}_H$, or ${\bf a'}$ is at least as large as the weight of ${\bf p}$. But that was the unique maximum, thus ${\bf p}$ is equal to either ${\bf b}_H$, or ${\bf a'}$. As ${\bf p}$ has a non-zero coordinate below $n/2$, ${\bf p}$ cannot be equal to ${\bf a'}$.

Let $i \le n/2$ be such that $f_i/|HM(i,m)|=:\lambda$ is maximal. Then $\lambda {\bf b}_H$ has at least $f_j$ in coordinate $j$ for every $j \le n/2$. 
Let us consider now a coordinate $k>n/2$ with $w(k)>0$.

If the family $\cF_i\cup\cF_k$ is trivially intersecting, then $f_k\le |HM(k,m)|$, while ${\bf b}_H$ and ${\bf a'}$ both have at least $|HM(k,m)|$ in coordinate $k$, thus so does $\lambda {\bf b}_H + (1-\lambda){\bf a'}$, completing the proof.

If the family $\cF_i\cup\cF_k$ is non-trivially intersecting, we can apply Lemma \ref{ketto}. It implies that there is a $\lambda'$ such that $(f_i,f_k)\le ((1-\lambda')|HM(i,k)|,\lambda'\binom{n}{k}+(1-\lambda')|HM(k,k)|)$. Coordinate $i$ shows that $((1-\lambda')|HM(i,k)|\ge \lambda |HM(i,m)|$. Since $|HM(i,m)|\ge |HM(i,k)|$, this implies that $\lambda \le 1-\lambda'$. Consider now coordinate $k$. We have \begin{equation}\label{equa}\tag{$\star$}
    f_k\le \lambda'\binom{n}{k}+(1-\lambda')|HM(k,k)|\le \lambda'\binom{n}{k}+(1-\lambda')|HM(k,m)|.
\end{equation} Since $|HM(k,m)|\le \binom{n}{k}$, increasing $\lambda'$ increases the right hand side of (\ref{equa}). Since $\lambda'\le 1-\lambda$, the right hand side is at most $(1-\lambda)\binom{n}{k}+\lambda|HM(k,m)|$, which is coordinate $k$ of $\lambda {\bf b}_H + (1-\lambda){\bf a'}$, completing the proof.

\end{document}